\documentclass[12pt]{article}
\usepackage{amssymb}
\usepackage{amsmath}
\usepackage{latexsym}
\usepackage{amsthm}
\usepackage{mathrsfs}
\usepackage{stmaryrd}
\usepackage{wasysym}
\usepackage[OT2,OT1]{fontenc}
\usepackage{graphicx}
\DeclareGraphicsExtensions{.eps, .jpg}
\usepackage[OT2,OT1]{fontenc}
\newcommand\cyr{%
\renewcommand\rmdefault{wncyr}%
\renewcommand\sfdefault{wncyss}%
\renewcommand\encodingdefault{OT2}%
\normalfont
\selectfont}
\DeclareTextFontCommand{\textcyr}{\cyr}
\DeclareMathAlphabet{\zap}{OT1}{pzc}{m}{it}

\newcommand{\Dye}{\Delta}

\newcommand{\pull}{\mbox{\cyr   sh}}

\newcommand{\CC}{\mathbb{C}}

\newcommand{\CP}{\mathbb C\mathbb P}
\newcommand{\RP}{\mathbb R\mathbb P}

\newcommand{\RR}{\mathbb R}

\newcommand{\ZZ}{\mathbb{Z}}

\newcommand{\ind}{\operatorname{ind}}
\newcommand{\push}{{\textcyr c}}

\newtheorem{thm}{Theorem}
\newtheorem{prop}{Proposition}

\newtheorem{lem}{Lemma}

\theoremstyle{definition}
\newtheorem*{defn}{Definition}
\theoremstyle{remark}
\newtheorem*{remark}{Remark}
\newtheorem*{acknow}{Acknowledgment}

\begin{document}
\title{Zoll Metrics, Branched Covers, and 
Holomorphic Disks}
\author{Claude LeBrun\thanks{Supported 
in part by  NSF grant DMS-0905159.} ~and
L.J. Mason\thanks{Supported in part by  FP6
Marie Curie RTN {\em ENIGMA} (MRTN-CT-2004-5652).}
}
\date{February 11, 2010}

\maketitle 
\begin{abstract}
We   strengthen our previous results \cite{lmzoll} regarding 
the moduli spaces  of Zoll metrics and Zoll projective structures 
on  $S^2$. In particular, we describe a concrete, open  condition 
which suffices to guarantee that a   
 totally real embedding  $\RP^2\hookrightarrow \CP_2$ 
 arises  from a unique Zoll projective structure on the $2$-sphere. 
Our methods ultimately reflect the special role such structures play  in the 
initial value problem for 
 the $3$-dimensional Lorentzian Einstein-Weyl equations. 
\end{abstract}

\bigskip

A  {\em Zoll metric} on a smooth manifold $M$  is a  Riemannian
metric $g$   whose geodesics are all simple closed curves of equal length. 
This terminology commemorates the  fundamental contribution of  Otto Zoll \cite{zoll}, 
who exhibited  an infinite-dimensional  family  of  such metrics   on 
$M=S^2$. It is easy to prove    \cite{beszoll} that a manifold admitting Zoll
metrics is compact and has finite fundamental group,
so the only two-dimensional candidates for $M$ are $S^2$ and $\RP^2$; 
conversely,  the standard metrics   on both of these surfaces are obviously 
Zoll.  However, 
  Green's proof \cite{grezoll} of the Blaschke conjecture
    shows that, after rescaling,  every  Zoll metric on $\RP^2$
  is actually a pull-back of the  standard one  via  
  some diffeomorphism.  By contrast, Zoll's examples show that 
   the situation for the $2$-sphere  
is fundamentally different. Indeed, 
in the decade following  Zoll's work, 
Funk \cite{funk}  gave a formal-power-series argument
indicating that, modulo isometries and rescalings, 
 the general  Zoll perturbation of the standard metric
on $S^2$  depends on  one {odd} function 
$f: S^2 \to \RR$. However, a rigorous proof of Funk's conjectural picture 
was only found half a century later, when  Victor Guillemin \cite{guillzoll}
brought the power of 
Nash-Moser 
 implicit function theorems to bear on the problem.

More recently, 
twistor techniques have given us   new insights into 
global aspects of the problem. Indeed, the present authors  have elsewhere
shown   \cite{lmzoll}  that  
Zoll surfaces can  in principle
be completely understood in terms of  families of holomorphic disks in $\CP_2$. 
These same techniques are also  naturally adapted to the study 
of more general {\em Zoll projective structures}.  Recall that a
 projective structure is by definition 
an equivalence class $[\nabla ]$ of 
affine connections $\nabla$ on a manifold $M$, where two
connections are declared to be equivalent iff they have
the same geodesics,   considered  as {unparameterized} curves. 
A projective structure is said to be Zoll iff its  geodesics (again, 
as unparameterized curves)  are all embedded circles. It can
then be shown   \cite{grogro,lmzoll} that a  Riemannian metric $g$ 
on a compact surface $M$ is   Zoll  iff 
the equivalence class  $[\nabla ]$ of  its Levi-Civita 
connection   is a Zoll projective structure. 
A compact surface $M$ can admit a Zoll projective structure $[\nabla ]$ 
iff it is diffeomorphic to $S^2$ or $\RP^2$; and,  as in the Riemannian case, 
any Zoll projective structure on $\RP^2$ is actually the standard one, 
pulled back via some self-diffeomorphism of $\RP^2$. Our proof  of 
this last assertion  \cite{lmzoll} hinged on  the fact  that the complex structure of 
$\CP_2$ is unique \cite{yau} up to biholomorphism.

We now summarize our previous results \cite{lmzoll} regarding the 
 the case of  $M=S^2$. 
Given a smooth Zoll projective structure $[\nabla ]$ on $M$, its space of 
unoriented 
geodesics $N\approx \RP^2$ has a natural
embedding in $\CP_2$ as
a totally real submanifold, in  a manner which is completely determined 
up to a projective linear transformation;  for example, the usual 
projective structure induced by  the  standard 
``round''  metric 
corresponds to  a ``real linear'' embedding $\RP^2\hookrightarrow \CP_2$.
Each point $x\in M$ determines an embedded holomorphic
disk $\Delta_x\subset\CP_2$ with $\partial \Delta_x\subset N$,
and the 
 relative homology class $[\Delta_x]$ of any such  disk 
generates
$H_2 (\CP_2, N; \ZZ)\approx \ZZ$. 
These disks meet $N$ only along their boundaries, and their interiors
foliate $\CP_2- N$. The family of disks $\Delta_x$  moreover sweeps out an entire
connected component  in the moduli space of holomorphic disks 
$(D^2, \partial D^2) \to (\CP_2, N)$. If the family of disks $\{ \Delta_x~|~x\in M\}$  is 
known, the projective structure $[\nabla ]$ can then be completely reconstructed;
namely, given a point $z\in N$, the set
$${\mathfrak C}_z = \{ x\in M~|~ z\in \partial \Delta_z \}$$
is a geodesic of $[\nabla ]$, and every geodesic arises in this way. 

The construction  proceeds by first creating an abstract
complex surface,  and then showing that it must be biholomorphic to $\CP_2$. 
In the  process, the bundle of orientation-compatible almost-complex structures over 
$M=S^2$ is identified with the   complement $\CP_2 - N$ of
the relevant totally real submanifold $N$. 
If there is an orientation-compatible complex structure ${\zap J}$ on 
$M$ which is parallel with respect to some  torsion-free  connection  $\triangledown \in [\nabla ]$,
then the graph of ${\zap J}$ 
becomes  a  holomorphic curve ${\mathcal Q}\subset \CP_2 -N$. 
For homological  reasons, this curve must  be a non-singular conic, 
and so may be put in the standard form 
\begin{equation}
\label{conic}
z_1^2 + z_2^2 + z_3^2 =0
\end{equation}
by making a suitable choice of homogeneous coordinates on $\CP_2$.  
Notice that this happens precisely when there is a conformal
structure $[g]$ on $M$ for which  $\triangledown$
is a compatible Weyl connection. 
If there is actually a Zoll metric $g$
with   Levi-Civita connection $\triangledown \in [\nabla ]$, 
then 
the totally real submanifold $N\subset \CP_2$ is moreover  
{\em Lagrangian} with respect to the sign-ambiguous 
symplectic form $\Omega = \Im m ~\Upsilon$ on $\CP_2 - {\mathcal Q}$, where 
\begin{equation}
\label{oops}
 \Upsilon = \pm ~
\frac{z_1 ~dz_2\wedge dz_3 + z_2 ~dz_3 \wedge dz_1 + z_3~ dz_1\wedge dz_2}{
{\sqrt{(z_1^2 + z_2^2 + z_3^2)^3}}
} ~ . 
\end{equation}

In the converse direction, one would like to assert that the totally
real submanifold $N\subset \CP_2$ can be chosen essentially arbitrarily, and 
that each such choice uniquely determines a Zoll projective structure $[\nabla ]$
on $M=S^2$. But while  our previous results in this direction may have
been  conceptually suggestive, they were technically  crude  in 
 important respects. Indeed, 
using an elementary 
inverse-function theorem argument, 
we merely showed in \cite{lmzoll} that every  $N\subset \CP_2$ which is $C^{2k+5}$ 
close to the standard ``real linear'' $\RP^2$ in the  topology actually arises from a 
$C^k$ Zoll projective structure $[\nabla ]$, and that this projective structure is unique
among those that are   close to the standard ``round'' projective
structure.  By contrast, the rest of the story was quite clean; the choice of a  
reference conic  ${\mathcal Q}\subset \CP_2$
disjoint from such an $N$ then gives rise to a conformal structure $[g]$ on $M=S^2$
for which the Zoll  projective structure $[\nabla ]$ is represented by a unique $[g]$-compatible 
Weyl connection $\triangledown \in [\nabla ]$, and this Weyl connection 
is the Levi-Civita connection of a Zoll metric $g\in [g]$ iff $N$
is Lagrangian with respect to  the sign-ambigious symplectic form $\Omega$. 
Still, it must be admitted that 
 our previous results  remain esthetically unsatisfactory in two  essential ways:
we neither provided an effective  condition on $N\subset \CP_2$ sufficient for the existence
of an associated  family of holomorphic disks, nor  proved the 
 uniqueness of this family when it  does exist.

The present article will address these issues 
by proving  global existence and
uniqueness results for 
 holomorphic disks; see  Theorems \ref{snap}, \ref{crackle}, and \ref{pop} below. 
For the sake of  clarity, our discussion is set  
almost entirely 
 in  the smooth
($C^\infty$) context. While our present methods certainly afford us this luxury, 
the interested  reader may nonetheless  wish to verify   that most of our arguments
can in fact  be carried out with much less regularity. 
The reader may also find it interesting to compare and
contrast our uniqueness results with the rather different ones
found in \cite{rochon}. 

We now begin by  fixing  the standard non-singular conic $\mathcal Q\subset \CP_2$  given by 
(\ref{conic}). 
Of course, any two non-singular conics are actually projectively equivalent, 
but our  conventional choice of $\mathcal Q$  has the nice additional  feature  that  it 
is manifestly invariant under an anti-holomorphic involution
\begin{eqnarray}
{\mathfrak c}: \CP_2 & \longrightarrow & \CP_2 \label{conjugate} 
 \\
~[  z_1 ,   z_2 ,   z_3]  
&\longmapsto   & [\bar{z}_1:\bar{z}_2:\bar{z}_3] \nonumber 
\end{eqnarray}
whose fixed-point set is disjoint from $\mathcal Q$.  This fixed-point set will 
henceforth be called  the {\em standard}  $\RP^2 \subset \CP_2$. 

Next,  notice that a  projective line ${\mathcal A} \subset \CP_2$ is  tangent to ${\mathcal Q}$
iff it is given by 
$$a_1 z_1 + a_2 z_2 + a_3 z_3 =0$$
for an element $[a_1:a_2:a_3]$ of the dual projective plane $\CP_2^*$  satisfying
\begin{equation}
\label{dual}
a_1^2 + a_2^2 + a_3^2 =0.
\end{equation}
When this happens,   the point of tangency is then  given by 
$$[z_1:z_2:z_3]= [a_1:a_2:a_3].$$ 
Also notice  that if 
$p= [p_1:p_2:p_3]$ belongs to the complement of $ {\mathcal Q}$ in $\CP_2 $, 
there  are always exactly two tangent lines of $\mathcal Q$ which pass through $p$;
 indeed,  the incidence equation 
$$a_1 p_1 + a_2 p_2 + a_3 p_3 =0$$
for  $[a_1:a_2:a_3]$  describes a line in the dual projective plane $\CP_2^*$
which is {\em not}  tangent to the dual conic (\ref{dual}), and which  therefore, 
by  B\'ezout's theorem, 
must  meet
the  dual conic  in   two distinct points.

The standard $\RP^2 \subset \CP_2$ is an example of 
a totally real submanifold. If $Z$ is a complex manifold with 
integrable almost-complex structure $J$, 
 recall that a 
differentiable  submanifold $S\subset Z$  is said to be (maximally)   {\em totally
real} if $TZ|_S = TS \oplus J(TS)$.
We  now introduce a  special  class of totally real surfaces
in $\CP_2$ that  will
be central  to our discussion. 

\begin{defn} A   compact  connected smoothly  embedded 
$2$-manifold $N \subset \CP_2$ will be called a  {\em docile surface}
if
\begin{itemize}
\item $N$ is a totally real  submanifold of $ \CP_2$;
\item $N$ is disjoint from the conic ${\mathcal Q}$ defined by (\ref{conic});
and 
\item   $N$ is transverse to each   tangent projective line (\ref{dual}) 
of  the conic ${\mathcal Q}$. 
\end{itemize}
\end{defn}

For example, the standard $\RP^2 \subset \CP_2$ is docile. Indeed, the
transversality condition is satisfied in this case because  a projective line is tangent to 
this standard $\RP^2$ iff it is invariant under complex conjugation, whereas the 
two tangent lines  of  $\mathcal Q$ which pass through a real point 
$p \in \RP^2$ are, by contrast,   interchanged by  the anti-holomorphic 
involution  $\mathfrak c$ defined above in (\ref{conjugate}). 

Since the condition of docility is obviously open in the $C^1$ topology, 
any  small perturbation of the standard $\RP^2$ will also be a 
docile surface. In the converse direction, we have the following result: 

\begin{lem}
\label{primo} 
Let $N\subset \CP_2$ be a docile surface. Then $N$
is diffeomorphic to $\RP^2$, and is isotopic
 to the standard  $\RP^2\subset \CP_2$
 through a family of 
other docile surfaces. 
\end{lem}
\begin{proof}
The argument will proceed by systematically exploiting    the map 
\begin{eqnarray*}
\Pi: \CP_1 \times \CP_1 &\to& \CP_2\\
([u_1:u_2], [v_1:v_2])&\mapsto&
[i(u_1v_1+u_2v_2):u_1v_1-u_2v_2:u_1v_2+u_2v_1]~,
\end{eqnarray*}
which  is a $2$-to-$1$ branched cover, ramified over the conic $\mathcal Q$. Indeed, 
notice that 
the diagonal $\mathcal D$, explicitly given   by 
$u_1v_2 - v_1 u_2 =0$,  is sent  bijectively to the conic $z_1^2+z_2^2+z_3^2=0$. 
Moreover,  each factor line $u= \mbox{const}$ or $v= \mbox{const}$ is
sent isomorphically to a tangent line $\mathcal A$ of $\mathcal Q$, and 
every tangent line conversely occurs in each of these families. 
Since $\Pi$ has degree two, and since exactly two tangent lines pass
through each  $p\in \CP_2 - {\mathcal Q}$, it  follows that
$\Pi$ is  actually unramified away from ${\mathcal Q}$. 
We note in passing   that the 
``anti-diagonal''  $\overline{\mathcal D}$, defined by
$[u_1:  u_2 ] = [ -\bar{v}_2:\bar{v}_1]$,  is therefore  a $2$-to-$1$  cover of the 
standard $\RP^2 \subset \CP_2$ defined  by   $z_j=\bar{z}_j$, $j=1,2,3$. 

Now suppose that  $N\subset \CP_2$ is a docile surface, and let 
$\tilde{N}=
\Pi^{-1}(N)$ be its pre-image in $\CP_1 \times \CP_1$. 
Since $N \cap {\mathcal Q}= \varnothing$,  it
follows that $\tilde{N}$ is a smooth compact surface in $\CP_1 \times \CP_1$,
and the induced map
$\tilde{N} \to N$  
is a  two-to-one submersion. 
In particular, $\tilde{N}$ has at most two connected components. 
On the other hand, the transversality hypothesis guarantees that 
$\tilde{N}$ is transverse to each factor line  $u= \mbox{const}$ or $v= \mbox{const}$,
so each factor  projection 
$$\varpi_j|_{\tilde{N}}: \tilde{N} \to \CP_1,~~~j= 1, 2,$$
 is  a submersion, and hence  a covering map on each connected component. 
 Since $\CP_1$ is simply connected, 
  both projections are therefore diffeomorphisms on each connected component, 
  and each component of $\tilde{N}$ is therefore the graph of
  some diffeomorphism $\varphi : \CP_1 \to \CP_1$. Moreover, since 
  $\tilde{N}\cap {\mathcal D}= \varnothing$,  the relevant 
  $\varphi$ cannot have fixed points, and is therefore  homotopic to the 
  antipodal map 
 \begin{eqnarray*}
{\mathfrak a} : \CP_1 &\longrightarrow & \CP_1 \\
  ~[ z_1 : z_2 ] & \longmapsto & [ - \bar{z}_2 : \bar{z}_1 ]
\end{eqnarray*}
  via the familiar  geometric construction of pushing $\varphi (z)$ away from 
  $z$ along  great circles. Consequently, each such diffeomorphism $\varphi$ has 
  degree $-1$, and so is  orientation-reversing. 
  
  Thus,  if $\tilde{N}$
  were disconnected, it would  have to  be 
  the   union of two disjoint  graphs of orientation-reversing diffeomorphisms
  $\varphi_1, \varphi_2 : \CP_1 \to \CP_1$.  Since these two graphs would 
  be disjoint, we would have   $\varphi_1(u) \neq \varphi_2(u)$,
  and hence $(\varphi_2^{-1}\circ \varphi_1) (u) \neq u$,  
  for all $u \in \CP_1$. Thus 
    $\varphi_2^{-1}\circ \varphi_1$  would   be fixed-point-free, and hence   would also have 
    have degree $-1$. But since 
  $$\deg   (\varphi_2^{-1}\circ \varphi_1) = (\deg \varphi_2)^{-1}(\deg \varphi_1) = (-1)^2 = + 1,$$ 
 this is  clearly a contradiction. 
  
 Hence  $\tilde{N}$ is  connected, and  therefore the
  graph of a single fixed-point-free, 
   orientation-reversing diffeomorphism $\varphi :  \CP_1 \to \CP_1$.
   However, by construction,    $\tilde{N}= \Pi^{-1}(N )$ is  invariant under the 
holomorphic  involution 
 \begin{eqnarray*}
\varrho : \CP_1 \times \CP_1 &\longrightarrow & \CP_1 \times \CP_1\\
 (u,v) & \longmapsto & (v, u)
\end{eqnarray*}
and every  $(u ,\varphi (u))$ must therefore also be expressible as $(\varphi (v), v)$.
   Thus  $\varphi= \varphi^{-1}$, and   $\varphi^2 = \mbox{id}$. 
   By projection to the first factor, the action of the non-trivial deck transformation 
   on $\tilde{N}$ can therefore be identified with the action of $\varphi$
   on $\CP_1$,  and the 
   quotient $N$ of $\tilde{N}$ by 
this deck transformation can therefore  be identified with
  $\CP_1/\langle \varphi \rangle$. But 
 this is a smooth compact surface with  fundamental group $\ZZ_2$, and so  necessarily 
diffeomorphic to $\RP^2= \CP_1 / \langle {\mathfrak a} \rangle$. Lifting some  diffeomorphism
${\psi}_0: \CP_1 / \langle {\mathfrak a} \rangle \to \CP_1/\langle \varphi \rangle$ to an orientation-preserving
diffeomorphism  $\psi : \CP_1 \to \CP_1$ of universal covers now  shows that we must have
$$\varphi = \psi \circ {\mathfrak a} \circ  \psi^{-1}$$
for some orientation-preserving diffeomorphism $\psi$ of $\CP_1$.

This conclusion  can now be  reverse-engineered. Indeed, 
 given any orientation preserving diffeomorphism $\psi : \CP_1 \to \CP_1$, 
 the graph of the orientation-reversing involution 
$\varphi = \psi \circ {\mathfrak a} \circ  \psi^{-1}$
necessarily  projects, via  $\Pi$, to a  docile surface in $\CP_2$. 
However, the group of orientation-preserving diffeomorphisms
of $\CP_1$ is  connected, since it acts transitively on the (connected) space of conformal 
metrics on $S^2$, 
with (connected) isotropy subgroup  $PSL(2, \CC )$. 
It therefore follows that the space of docile surfaces
in $\CP_2$ is also connected.  In particular, any given 
docile surface $N$ may be smoothly deformed, via a family of docile surfaces, 
into the  standard $\RP^2\subset \CP_2$. 
    \end{proof}

\begin{lem} 
\label{secondo} 
Let $N\subset \CP_2$ be any docile surface. Then 
the homomorphism 
$H_2 (\CP_2 , N) \to  \ZZ$ given by homological intersection with 
$[{\mathcal Q}] \in H_2 (\CP_2 - N)$
is an isomorphism.  In  particular, $H_2 (\CP_2 , N) \cong \ZZ$. 
\end{lem}
\begin{proof} Let us first observe that 
there is a homeomorphism $\Psi: \CP_2\to \CP_2$
which sends $N$ to $\RP^2$,  and  $\mathcal Q$ to itself.
Indeed, let $\psi : \CP_1 \to \CP_1$
be a  diffeomorphism such that $\tilde{N}$ is the
graph of $\psi\circ {\mathfrak a} \circ \psi^{-1}$, and then define  
$\tilde{\Psi} : (\CP_1\times \CP_1) \to (\CP_1\times \CP_1)$ to be 
$\psi \times \psi$.  Then $\tilde{\Psi}$
send  $\bar {\mathcal D}$ to  $\tilde{N}$,  and 
${\mathcal D}$ to itself, while commuting 
 with the involution $\varrho$ given by $\varrho (u, v) \mapsto (v, u)$. 
Since $\varrho$  acts transitively on the fibers of $\Pi$, and 
since $\Pi :  \CP_1 \times \CP_1 \to \CP_2$ is a quotient map, 
there is a   unique homeomorphism $\Psi : \CP_2 \to \CP_2$ 
such that $\Pi \circ \tilde{\Psi} = \Psi \circ \Pi$. Moreover, this homeomorphism 
$\Psi$  then  sends $N$ to $\RP_2$,  and  $\mathcal Q$ to
  $\mathcal Q$, as promised.

Thus $H_2 (\CP_2 , N ) \cong H_2 (\CP_2  , \RP^2 )$, and  
we only really need  to check the claim when  
$N = \RP^2$. 
However,  $\CP_2$ is  the union of
a $2$-disk bundle $X\to \RP^2$ and a $2$-disk bundle 
$Y \to {\mathcal Q}$, identified along their  boundaries. (This 
 can  be checked by hand \cite{lmzoll}, 
but it can more elegantly be  deduced  \cite{grozi}  from the fact that there is 
a  cohomogeneity-one action of $SO(3)$ on $\CP_2$, 
with   exceptional orbits  $\RP^2$ and $\mathcal Q$.)  
Thus $H_2 (\CP_2 , \RP^2)\cong   H_2 (\CP_2 , X)$
by homotopy equivalence, whereas 
$H_2 (\CP_2 , X) \cong H_2 (Y , \partial Y )$ by excision, and 
$H_2 (Y,  \partial Y)\cong H^2 (Y)$ by Lefschetz-Poincar\'e duality. 
Since ${\mathcal Q}$ is a deformation retract of $Y$,
it   follows that $H_2 (\CP_2 , N )\cong H^2 ({\mathcal Q}) = \ZZ$. 
Moreover, since  $[\mathcal Q]$ generates $H_2 (Y)\cong \ZZ$, 
Lefschetz-Poincar\'e duality 
guarantees that homological intersection with  $[\mathcal Q]$ is an 
isomorphism $H_2(\CP_2, N)\to \ZZ$. 
\end{proof}

\begin{remark} The homeomorphism $\Psi$ constructed above 
will typically fail to be smooth along $\mathcal Q$. However, it is 
not difficult to modify $\Psi$ to  produce 
a self-diffeomorphism of $\CP_2$ with all the  properties in question. We leave this 
exercise as a challenge to the 
 interested reader. 
\end{remark}

\begin{lem}
\label{trio}
 Let $N \subset \CP_2$ be any docile surface, 
and let $\tilde{N}\subset \CP_1 \times \CP_1$ be its inverse image under 
$\Pi$. Then $\CP_1 \times \CP_1$ admits a $\varrho$-invariant 
K\"ahler metric $h$  for which  $\tilde{N}$ is Lagrangian. 
This metric can  be chosen so that its 
K\"ahler form $\omega$ represents $2\pi c_1 (\CP_1 \times \CP_1 )$
in deRham cohomology,  and  if  $N$ is smoothly varied
through a family of other docile surfaces, a corresponding family of such 
K\"ahler metrics  can moreover be chosen so as to depend smoothly on the given 
parameters. 
\end{lem}
\begin{proof}
The construction is a variant of one used in \cite[Lemma 2]{lmew}. 
Express $\tilde{N}$ as the graph 
 of  a smooth  orientation-reversing involution  
 $\varphi : \CP_1 \to \CP_1$. Let $\alpha$ be the  area form of the standard
 unit-sphere metric on 
  $S^2=\CP_1$.  Then the  area form 
$$\check{\omega} = \frac{\alpha - \varphi^*\alpha}{2}$$
satisfies  $\varphi^*\check{\omega} = -\check{\omega}$, and 
is the K\"ahler form of a unique K\"ahler metric $\check{h}$ on $\CP_1$. 
We now define a K\"ahler metric
 on $\CP_1 \times \CP_1$ by setting 
 $$h=  \varpi_1^*\check{h} + \varpi_2^*\check{h}$$
where $\varpi_j  : \CP_1 \times \CP_1 \to \CP_1$, $j=1,2$,  are the factor projections. 
Since the restriction of the associated  K\"ahler form $\omega$
 to the graph of $\varphi$ is 
 $$
 (\mbox{id}\times \varphi)^* ( \varpi_1^*\check{\omega} + \varpi_2^*\check{\omega}) = 
 \check{\omega} + \varphi^*\check{\omega} =  \check{\omega}- \check{\omega}=0 , 
 $$
 the graph 
 $\tilde{N}$ is Lagrangian with respect to $\omega$. 
 Moreover, $h$ is invariant under the interchange of the 
 factors of $\CP_1\times \CP_1$, and the K\"ahler class
 of $h$ is obviously given by 
  $[\omega]=  2\pi c_1 (\CP_1 \times \CP_1)$, since $\omega$ has integral $4\pi$
  on either 
  factor $\CP_1$.  Finally, this construction can be uniformly applied 
  to a smooth family $N_t$ of docile surfaces by replacing 
  $\varphi$ with a corresponding of smooth family 
  $\varphi_t$ of smooth involutions,  and the  
 corresponding family 
 $h_t$  of K\"ahler metrics will then manifestly  depend smoothly on 
the  parameter  $t$. \end{proof}

\begin{defn}
 Let  $D^2$ denote 
the closed unit disk  in $\CC$, and let $Z$ be any complex manifold. 
A continuous map $f: D^2 \to Z$
will be called a {\em parameterized holomorphic disk} in $Z$
if $f$ is holomorphic in the open unit disk $\mathring{D}^2 = D^2 - \partial D^2$.
If, in addition, $f (\partial D) \subset W$ for a specified  subset
$W\subset Z$, we will sometimes say  that $f$ is a 
 {parameterized holomorphic disk} in $(Z,W)$. 
\end{defn}

\begin{prop} \label{discus} 
Let $N \subset \CP_2$ be any docile surface, and 
suppose that   $f$ is  a parameterized 
holomorphic disk in $(\CP_2 , N )$
whose relative homology class  $[f]$ generates $H_2(\CP_2, N) \cong \ZZ$. 
Then $f$ is a smooth embedding,  $f(D^2 )$ meets  $N$
only along 
$f(\partial D^2)$, and 
$f(D^2)$  meets $\mathcal Q$  transversely,  in a single point.
\end{prop}
\begin{proof}
Because $f|_{\partial D^2}$ takes values in the totally real
submanifold $N$, and the latter is assumed to be a submanifold of class  $C^{\infty}$, the holomorphic map $f$ must actually be smooth up to the boundary 
\cite{alibaro,chirka}.  Since we have also assumed that  $[f]$ generates 
$H_2(\CP_2, N) \cong \ZZ$,
its homological intersection number with ${\mathcal Q}$ must be $1$ by 
 Lemma \ref{secondo}, and the disk $f$ can therefore  only geometrically 
intersect $\mathcal Q$ transversely, in a single point,  
 since every geometric  intersection of  distinct holomorphic curves makes a
 contribution
with   positive multiplicity toward  their total homological intersection number. 

Now, 
because $\Pi : \CP_1 \times \CP_1\to \CP_2$ is a $2$-to-$1$ branched cover,
ramified only  at $\mathcal Q$, path-lifting of the null-homotopic 
 circles $f(re^{2i\theta})$ in $\CP_2 - {\mathcal Q}$ allows us
to construct a continuous lift  $\tilde{f}: D^2 \to \CP_1\times \CP_1$
with $\Pi(\tilde{f}(\zeta))= f(\zeta^2)$. Since this lift is moreover holomorphic 
away from the origin, it then  follows that $\tilde{f}$ is also holomorphic across the 
origin  by the Riemann removable singularities theorem. We thus obtain a
parameterized 
holomorphic disk $\tilde{f}: D^2 \to \CP_1\times \CP_1$ which 
 sends $\partial D^2$ to $\tilde{N} = \Pi^{-1} (N)$ and 
$0$ to the unique   $\tilde{p}\in {\mathcal D}$ with $\Pi (\tilde{p})= p$, while 
 satisfying  $\tilde{f}(-\zeta) = \varrho (\tilde{f}(\zeta))$, where $\varrho$ is once again the 
 involution $(u,v)\mapsto (v,u)$ of  $\CP_1 \times \CP_1$.  Also note  that, 
 given any $\tilde{f}$ satisfying these properties, one may conversely  construct a 
 disk  $f$  in $\CP_2$ with the desired properties by setting 
 $f(\zeta ) = \Pi (\tilde{f}(\pm \sqrt{\zeta}))$, since  this well-defined 
 map is obviously   continuous,  and is holomorphic away from the
 origin. 
 
 Now, given a holomorphic disk $f$ representing the generator of 
  $H_2(\CP_2, N)$,  the  associated    disk $\tilde{f}$ will 
  then represent the generator 
  of $H_2 (\CP_1 \times \CP_1 , \tilde{N})\cong \ZZ$. 
  Indeed, since $\tilde{N}$ is homotopic to the anti-diagonal
  $\overline{\mathcal D}$ in $(\CP_1 \times \CP_1)- {\mathcal D}$, 
  the long exact sequence 
  $$\cdots \to H_2 (\tilde{N}) \to H_2 (\CP_1 \times \CP_1 ) \to H_2 (\CP_1 \times \CP_1 , 
  \tilde{N}) \to 
  H_1 (\tilde{N}) \to \cdots $$
  implies that the diagonal  class $[\mathcal D]$ will necessarily represent
   twice the generator of 
  $H_2 (\CP_1 \times \CP_1 ,  
  \tilde{N})$. However,  $\Pi (\mathcal D )$ is the degree-$2$ 
  holomorphic curve $\mathcal Q$, so 
  $\Pi_* : H_2 (\CP_1 \times \CP_1 , 
  \tilde{N}) \to H_2 (\CP_2 , N )$  is therefore  the homomorphism $\ZZ\to \ZZ$ given by multiplication by $2$. 
  But $\Pi_* ([\tilde{f}])= 2 [f]$ 
   by construction, so it  follows that $[\tilde{f} ]$ is indeed
   the generator of $H_2 (\CP_1 \times \CP_1 , \tilde{N})$, as claimed. 
   In particular, since $\tilde{N}$ is Lagrangian with respect the 
   K\"ahler form $\omega$ of the metric $h$ constructed in Lemma 
   \ref{trio}, we must have
   $$
   \int_{D^2} \tilde{f}^* \omega = \frac{1}{2} \int_{\mathcal D} \omega = 
    \frac{1}{2} \int_{\mathcal D} 2\pi c_1 =  \frac{1}{2} 2\pi (4) = 4\pi ~.
   $$
   
   Next, by making precise an argument
    previously sketched in  \cite[Lemma 3]{lmew}, we will show 
     that $\tilde{f}$ must be a topological 
 embedding. 
 Indeed, consider the abstract oriented  $2$-sphere obtained by taking the double
$D^2 \cup \overline{D^2}$, where the two copies of the disk are identified along the
boundary, $D^2$ is given the usual orientation coming from the
unit disk in $\CC$, and $\overline{D^2}$ is given the opposite orientation. 
Given $f$ as above, we can then construct a  continuous map 
$F: D^2 \cup \overline{D^2}\to \CP_1$, defined to equal 
$\varpi_1 \circ\tilde{f}$ on $D^2$, and to equal $\varphi^{-1}\circ \varpi_2 \circ\tilde{f}$
on $\overline{D^2}$. Note that  ${F}$ is 
actually smooth when  restricted to either $D^2$ or $\overline{D^2}$, and that
it is orientation-preserving at every regular point of
 the interior of either $2$-disk hemisphere. 
 The complex dilatation $\mu$ of $F$  is bounded by that of $\varphi^{-1}$, so $F$ is consequently 
quasi-regular in the sense of \cite{rickman}. However, 
via the measurable Riemann mapping theorem, 
this implies
 \cite[section VI.2.3]{lehvir} that $F= G\circ H$ 
for some  quasi-conformal homeomorphism $H: D^2 \cup \overline{D^2}\to \CP_1$
and some holomophic map
$G: \CP_1 \to \CP_1$.  

However,  equipping $\CP_1$ and $\CP_1\times \CP_1$ respectively with the K\"ahler forms
 $\check{\omega}$ and $\omega$ used in the proof of 
 Lemma \ref{trio}, we have 
$$\int_{D^2 \cup \overline{D^2}}F^*\check{\omega}=
\int_{D^2}\tilde{f}^*(\varpi_1^*\check{\omega}+ \varpi_2^* \omega_2) = \int_{D^2}\tilde{f}^*\omega
= 4\pi = \int_{\CP_1}\check{\omega}$$
and it therefore follows that the piece-wise smooth map ${F}$ has degree $1$. 
The holomorphic map $G: \CP_1 \to \CP_1$ is therefore a biholomorphism, and 
$F$ is therefore a quasi-conformal homeomorphism. 
In particular,  $\varpi_j\circ  \tilde{f}$ must be  a homeomorphism for $j=1,2$, 
and 
$(\varpi_1 \circ\tilde{f}) (\mathring{D}^2)$ must be disjoint from 
 $(\varphi^{-1}\circ \varpi_2 \circ\tilde{f})(\mathring{D}^2)$. In particular, 
 the graph $\tilde{N}$ of $\varphi$ is disjoint from $\tilde{f} (\mathring{D}^2)$, 
  and hence ${N}$  is disjoint from ${f} (\mathring{D}^2)$,  too. 

To finish the proof, 
we will now show that the smooth maps  
$$\varpi_j\circ  \tilde{f}: D^2 \to \CP_1, ~~ j=1,2,$$
 are  actually smooth embeddings. Of course, since we already know that they 
 are   homeomorphisms, and since they are manifestly 
 holomorphic on the interior $\mathring{D}^2$ of the disk, 
 it only remains to show that these maps are immersions along $\partial D^2$. 
However, since $\tilde{N}\subset \CP_1 \times \CP_1$ is a smooth, totally
real submanifold, and since $\tilde{f} (\partial D^2) \subset \tilde{N}$, it follows 
 \cite{alibaro} that 
the non-constant
holomorphic map $\tilde{f}$ cannot  be constant to infinite order at any 
boundary point. Moreover,  since $\tilde{N}$ is the graph
of a diffeomorphism $\varphi: \CP_1\to \CP_1$, neither of the factor maps  
$\varpi_j \circ \tilde{f}$, $j=1,2$,   can  be constant to infinite order
at a boundary point,  either. Thus, setting  $\push= \varpi_1 \circ \tilde{f}$, 
letting  $\zeta_0\in \partial D^2$ be any boundary point of the disk, 
and equipping $\CP_1$ with a local complex coordinate  $z$ centered at $\push(\zeta_0)$, 
there must be an integer $k\geq 1$ such that 
$$
\push(\zeta) =  a (\zeta-\zeta_0)^k + O(|\zeta-\zeta_0|^{k+1})
$$
for some $a\neq  0$, 
since the $C^{\infty}$ function $\push$ satisfies the Cauchy-Riemann
equations up to the boundary. If $k\geq 2$, it therefore follows  that,
in the fixed coordinate system, $\push(\mathring{D}^2)$ must 
contain a punctured half-disk
$$0< |z | \leq  \epsilon , ~~ \Re e (e^{-i\theta} z) \geq 0$$
because the boundary of a wedge-shaped subregion 
$$|\zeta -\zeta_0| <  \delta , ~~ \left|\arg (\zeta-\zeta_0) -\theta\right| < \frac{\pi }{3k}$$
of $D^2$ 
is mapped by the smooth homeomorphism $\push$ to a piece-wise $C^1$ Jordan curve with internal 
break-angle $4\pi/3 > \pi$ at $\push(\zeta_0)$. 
However, 
the same argument can also be applied
to the holomorphic map $\varpi_2 \circ \tilde{f}$; thus, if 
 $k\geq 2$, even after 
composing with  the orientation-reversing  diffeomorphism $\varphi^{-1}$,
the boundary 
of an analogous wedge-shaped region would still be sent to 
a piece-wise $C^1$ 
Jordan curve  whose internal break angle would have  absolute value $> \pi$ at the origin,
and    the image of  $\mathring{D}^2$ under the 
smooth map $\pull = \varphi^{-1} \circ \varpi_2 \circ \tilde{f}$
would  therefore also contain a punctured half-disk
$$0< |z | \leq \varepsilon , ~~ \Re e (e^{-i\vartheta} z) \geq 0 .$$
Since any two such punctured half-disks must  meet,  $k\geq 2$  thus  implies  
that 
$$
\push (\mathring{D}^2) \cap \pull (\mathring{D}^2) \neq \varnothing , 
$$
contradicting the  previously  
established fact that $F$ is injective. Hence $k=1$, and, since $z_0\in \partial
 D^2$ is arbitrary,  
 $\push$ and $\pull$ are both smooth embeddings  
$D^2  \hookrightarrow \CP_1$. 
The parameterized holomorphic disk $\tilde{f}$ in  $(\CP_1 \times \CP_1, \tilde{N})$ is therefore smoothly embedded, and the 
 parameterized holomorphic disk $f$ in  $(\CP_2, N )$ is therefore 
   smoothly embedded, too.  
   \end{proof}

If $L \to D^2$ is
a complex line bundle over the disk, and if $\ell \to S^1$ is a 
 real line sub-bundle  of $L|_{\partial D^2}$, 
 recall \cite{mcsalt} that the Maslov index  $\ind (L,\ell )$
 is  obtained by trivializing $L$,
viewing $\ell $ as a map $\partial D^2 \to \RP^1$, and  declaring  $\ind (L,\ell )$ to be the 
winding number of this map; the resulting integer  is 
independent of the trivialization of $L$, and  amounts \cite{lebrsb} to the first Chern class
of the double of $(L,\ell )$.  More generally, 
if $V\to D^2$ is a rank-$r$ 
complex vector bundle, and if ${\zap v} \to S^1$
is a rank-$r$ real sub-bundle of $V|_{\partial D^2}$, 
 the Maslov index  $\ind (V, {\zap v})$ is defined 
to be the Maslov index of the associated  line-bundle pair $(\Lambda^rV, \Lambda^r{\zap v})$.

If $ Z$ is a complex manifold
 and
$ W\subset  Z$ is a totally real submanifold, 
the {\em total Maslov 
index} of a parameterized  holomorphic disk ${\zap f}$ in $( Z, W)$
is defined 
 to be  $\ind ({\zap f}^*T Z, ({\zap f}|_{\partial D^2})^*T W)$. 
 If the disk happens to be embedded, with image $\Dye = {\zap f}(D^2)$, 
 we  also define the  
{\em normal Maslov index} of $\Dye$ 
to be  $\ind (N, {\zap n})$, where $N=T Z/T\Dye$ is the normal bundle of the disk, and 
${\zap n}= T W/T\partial  \Dye$ is the relative normal bundle of its boundary. 
 Since the Maslov index  is 
 additive,   in the sense that 
$$
\ind (V_1 \oplus V_2, {\zap v}_1 \oplus {\zap v}_2) = 
\ind (V_1, {\zap v}_1) + \ind (V_2, {\zap v}_2), 
$$
and because $\ind (TD^2 , T\partial D^2 )= 2$, 
 the total Maslov index of  
an embedded holomorphic disk  equals its normal Maslov
index plus two.

\begin{prop}\label{masl}
Let $N \subset \CP_2$ be a docile surface, and 
let  $f$ be  a parameterized 
holomorphic disk in $(\CP_2 , N )$
whose relative homology class  represents the  generator of  $H_2(\CP_2, N)$. 
Then  $f$ has  total Maslov
index $3$, and its image $f(D^2)$ has  normal Maslov index
$1$.
\end{prop}
\begin{proof}
Given such a disk $f$, let $\tilde{f} : D^2 \hookrightarrow 
\CP_1 \times \CP_1$ be the branched lifting constructed in 
the proof of Proposition 
\ref{discus}. 
Since  the embedded holomorphic disk  $\tilde{f}(D^2)$
is  the graph of a diffeomorphism between 
domains in $\CP_1$ with smooth boundary,  it is  transverse to the first factor of 
$T\CP_1 \times T\CP_1$, and 
 its normal bundle can therefore be identified with 
 the restriction of 
$T\CP_1$ to its first-factor projection $\push (D^2) \subset \CP_1$;
this  simultaneously identifies
 the relative normal bundle 
of the boundary  with $(J\circ \push_*) (T\partial D^2)$. 
Taking an affine chart that contains $\push(D^2)$,  setting $\gamma = \push|_{\partial D^2}$,
and systematically exploiting the  coordinate trivialization of 
$T\CC$,  
the normal Maslov index of $\tilde{f}(D^2)$  therefore equals  the 
winding number of 
$$(i\gamma^\prime)^2: S^1 \to \CC -\{0\}.$$
Since $\gamma$ is isotopic to the  standard circle  $e^{i\theta} \mapsto e^{i\theta}$
in $\CC$, 
$(i\gamma^\prime)^2$ is homotopic  to $e^{i\theta} \mapsto e^{2i\theta}$,
and  the normal Maslov index of $\tilde{f}(D^2)$  consequently  equals $2$. On the other hand, 
$\tilde{f}(D^2)$ is a $2$-to-$1$ branched over of $f(D^2)$, so 
the corresponding winding number  for $f$ is only half as large,  and 
the normal Maslov index of ${f}(D^2)$  therefore equals $1$. 
The  total Maslov index of $f$ is therefore 
 $3$, as claimed. 
\end{proof}

\begin{thm} \label{snap} 
Let $N \subset \CP_2$ be any docile surface, and let $p\in {\mathcal Q}$
be any point of the  reference conic (\ref{conic}). Then there is a 
holomorphic disk in $(\CP_2 , N)$ which passes through $p$
and represents the generator of $H_2(\CP_2, N) \cong \ZZ$. 
Moreover, this disk is unique, modulo reparameterizations. 
\end{thm} 
\begin{proof}
By Lemma \ref{primo}, there exists  a smooth family of docile surfaces 
 $N_t$, $t\in [0,1]$,  such that $N_1=N$, 
 and
 such that 
 $N_0$ is the standard  linear $\RP^2\subset \CP_2$.

For the docile surface $N_0=\RP^2$, we can construct such 
a disk by using the complex conjugation map $\mathfrak c$ defined by (\ref{conjugate}). 
Indeed, since ${\mathfrak c}$ acts freely on ${\mathcal Q}$, the points
$p$ and ${\mathfrak  c}(p)$ are distinct, and hence joined by a unique 
projective line $\CP_1\subset \CP_2$.  Since $p$ and 
${\mathfrak  c}(p)$ are interchanged by ${\mathfrak c}$, this projective line must 
be invariant under complex conjugation ${\mathfrak c}$, and  it is therefore
the complexification of a unique real projective line 
$\RP^1\subset \RP^2$.  This $\RP_1$ divides the 
complex projective line $\CP_1$ into two disks, exactly
one of which contains the given point $p\in {\mathcal Q}$;
and since this disk meets $\mathcal Q$ transversely in a single point,  
its relative homology class generates $H_2 (\CP_2 , \RP^2)$. 
Conversely, the only disk with these properties is this half-projective-line. Indeed, 
if $\Dye$ is a disk in $(\CP_2 , \RP^2)$ which 
passes through $p\in {\mathcal Q}$ and 
represents the  generator of  $H_2 (\CP_2 , \RP^2)$, 
then $\Dye \cup {\mathfrak c}(\Dye)$ is a holomorphic curve by the 
reflection principle, and, by B\'ezout's theorem,  has degree one because
it intersects  the conic ${\mathcal Q}$ in two points. Thus, any 
holomorphic disk representing the generator of  $H_2(\CP_2, \RP^2)$ 
is one hemisphere of a $\frak c$-invariant complex projective line, with boundary
a real projective line $\RP^2 \subset \RP^2$; and there
is exactly one such hemisphere containing  $p$. We have thus proved both existence and
uniqueness  for the ``model'' docile surface $N_0=\RP^2$. 

We will now apply the continuity method to obtain an appropriate
disk for each $t\in [0,1]$. Let ${\zap E}\subset [0,1]$ be the set of 
$t$ for which such a disk exists; our goal is to show that ${\zap E}=[0,1]$. 
 We already know that ${\zap E}\neq \varnothing$, since 
$0\in {\zap E}$. Because $[0,1]$ is connected,
it therefore suffices to show that  ${\zap E}$ is both open and closed. 

To show that ${\zap E}$ is open, we appeal to  the perturbation theory of 
holomorphic disks. Indeed, suppose that such a disk $\Dye$ exists for
a certain value $\tau$ of $t$. By Proposition \ref{discus}, $\Dye$  is 
smoothly embedded, and by Proposition \ref{masl}, its
normal Maslov index is $1$. The double of the normal 
bundle of $\Dye$, in the sense used in  \cite[Theorem 3]{lebrsb}, is therefore
the ${\mathcal O}(1)$ line bundle over $\CP_1= \Dye \cup \overline{\Dye}$, and, since
$H^1(\CP_1, {\mathcal O}(1))=0$, the disk $\Dye$ is Fredholm regular.
The moduli space $M_\tau$ of nearby holomorphic disks in 
$(\CP_2, N_\tau)$ is therefore smooth, with tangent space 
$$T_{\tiny \Dye}M_\tau=H^0_{\RR} (\CP_1 , {\mathcal O}(1))$$
where the right-hand-side denotes the real-linear 
subspace of  $H^0(\CP_1 , {\mathcal O}(1))$ consisting
of sections which are real along $\RP^1\subset \CP_1$. 
Now observe that evaluation at two distinct points $p,q\in \CP_1$ gives
rise to an isomorphism $H^0(\CP_1 , {\mathcal O}(1))\to \CC^2$, 
and that evaluation at a single point $p\in \CP_1-\RP^1$ therefore gives
us a real-linear isomorphism $H^0_{\RR} (\CP_1 , {\mathcal O}(1))\to \CC$, as
may be seen by setting $q=\bar{p}$. 
Since $\Dye$ is transverse to $\mathcal Q$ at  $p\in \mathring{\Dye}$, it follows that 
the map $\varkappa_\tau : M_\tau\to \mathcal Q$ obtained by 
sending a disk to its  intersection with  $\mathcal Q$ has maximal rank at
$\Dye$, and so is a  diffeomorphism between a neighborhood of $\Dye\in M_\tau$
and a neighborhood of $p\in {\mathcal Q}$. This shows that,  at least locally, 
$\Dye$ is   the only  disk in the family passing
through the chosen point $p\in {\mathcal Q}$. Moreover, the Fredholm regularity 
of $\Dye$ guarantees that, for all $t$ in a neighborhood of $\tau$, there
is a corresponding  $2$-parameter family $M_t$ of embedded holomorphic disks
in $(\CP_2, N_t)$, and the corresponding 
intersection map $\varkappa_t: M_t \to {\mathcal Q}$
is a local diffeomorphism onto a neighborhood of $p$. Consequently,  
for values of $t$ in an interval about 
 $\tau$,
there is a unique smooth family  $\Dye_t$ of   such  holomorphic disks  passing through $p$
with $\Dye_\tau=\Dye$. In particular, ${\zap E}\subset [0,1]$ is open. 

To show that ${\zap E}$ is closed, we will now use a Gromov compactness argument. Indeed, 
suppose that $t_j$ is a sequence of values of $t\in [0,1]$ for which there exist
a corresponding sequence of holomorphic disks $\Dye_{t_j}$ in 
$(\CP_2, N_{t_j})$, each of which intersects ${\mathcal Q}$ transversely
in the single  point $p$; moreover, suppose that the numbers
$t_j$ converge to some $\tau \in  [0,1]$. 
Let $\tilde{\Dye}_{t_j}= \Pi^{-1}(\Dye_{t_j})$ be the corresponding 
ramified lifts to $\CP_1\times \CP_1$. Each of these disks then meets the diagonal 
${\mathcal D}= \Pi^{-1}({\mathcal Q})$ transversely in the unique  point
$\tilde{p}$ corresponding to  $p$. Consequently, each one represents the  
generator of $H_2 (\CP_1\times \CP_1, {\mathcal D})$. For $t\in [0,1]$,  
we now endow 
$\CP_1 \times \CP_1$ with the smooth family of $\varrho$-invariant 
 K\"ahler forms $\omega_t$ constructed 
in Lemma \ref{trio}; 
this family  is chosen so that $\tilde{N}_t$ is Lagrangian with respect to 
$\omega_t$, and all these forms $\omega_t$ belong to 
  the same cohomology class, 
with respect to which each factor $\CP_1$'s has area $4\pi$. 
By Moser stablity and the Weinstein 
tubular neighborhood theorem for Lagrangian submanifolds, the triples 
 $(\CP_1 \times \CP_1, {\tilde{N}}_t , \omega_t)$ are thus all symplectomorphic
 to the standard model  $(\CP_1 \times \CP_1, {\tilde{N}}_0 , \omega_0)$,
allowing us to think of the $\tilde{\Dye}_{t_j}$ as being a sequence of disks 
 in  $(\CP_1 \times \CP_1, {\tilde{N}}_0 , \omega_0)$ which   are pseudo-holomorphic 
 with respect to a sequence of $\omega_0$-compatible  almost-complex structures. 
By Gromov compactness \cite{grofrau,groye}, this sequence therefore has a subsequence
which  converges   to a (possibly singular) pseudo-holomorphic curve
${\zap X}$ in the same relative homology class, where  ${\zap X}$ is a 
union of holomorphic $\CP_1$'s and at most one holomorphic 
disk in $(\CP_1\times \CP_1 , {\tilde{N}}_{\tau })$.  Moreover, since 
${\zap X}$  is the Gromov limit of a sequence of $\varrho$-invariant  curves through 
$\tilde{p}$, 
it must also be $\varrho$-invariant and pass through $\tilde{p}$. 
 However, by construction,
$$\omega_{\tau }= \varpi_1^* \check{\omega}_{\tau }+ \varpi_2^* \check{\omega}_{\tau }$$
for some area form $ \check{\omega}_{\tau }$ on $\CP_1$. Since 
${\zap X}$ is invariant under the factor-switching involution 
$\varrho$ and represents the generator in $H_2(\CP_1\times \CP_1 , {\tilde{N}}_{\tau })$, 
$$2\int_{\zap X} \varpi_1^*\check{\omega}_{\tau } = 
\int_{\zap X} \varpi_1^*\check{\omega}_{\tau }+
\int_{\zap X} \varpi_2^*\check{\omega}_{\tau } = 
\int_{\zap X} \omega_{\tau }= 4\pi . $$
The area, with multiplicities, of the projection of ${\zap X}$ to either factor will thus be 
$2\pi$, or in other words,   half the area of the entire sphere. In particular, 
 neither factor projection   ${\zap X}\to \CP_1$ can  be onto. Consequently, 
  ${\zap X}$ cannot have
any compact irreducible components, and is 
therefore a disk $\tilde{\Dye}_{\tau }$. Since $\tilde{\Dye}_{\tau }$
is $\varrho$-invariant, it must, moreover,  be the ramified  lift of a disk $\Dye_{\tau }$ in
$(\CP_2, N_{\tau })$. Because $\tilde{\Dye}_{\tau }$ meets the diagonal 
$\mathcal D$ transversely in the single point  
$\tilde{p}$, the disk $\Dye_{\tau }$ consequently meets $\mathcal Q$ 
transversely in the single point $p$. In particular,  $\Dye_{\tau }$
passes through $p$ and represents the 
 generator of 
 $H_2(\CP_2, N_{\tau })$. Thus $\tau  \in {\zap E}$, and  ${\zap E}$ is
 therefore closed. 
 
 Since  ${\zap E}\subset [0,1]$
  has now been shown to be  non-empty, open, and closed, the connectedness of 
  the  interval implies that 
  ${\zap E}=[0,1]$. In particular, $1\in {\zap E}$, so there
 exists a holomorphic disk in $(\CP_2 , N)$ which passes through $p$
and represents the generator of $H_2(\CP_2, N)$, as claimed. 
It only remains for us to show that this disk is in fact unique. 

To prove uniqueness, we apply the continuity method in reverse. Suppose that 
$\Dye^\prime$ is any holomorphic disk in $(\CP_2 , N)$ which passes through $p$
and represents the generator of $H_2(\CP_2, N)$. By the same disk-perturbation
argument as before, there is a unique smooth family $\Dye_t^\prime$, $t\in (\tau, 1]$,   of 
holomorphic disks in $(\CP_2, N_t)$, meeting $\mathcal Q$ transversely
in the single point $p$, such that   $\Dye_1^\prime=  \Dye^\prime$ and such that  
$\tau \in [0,1)$ is minimal. The above compactness argument then allows
one to construct a limit disk $\Dye_\tau$, and perturbations of this
disk allow one to extend the family across $t=\tau$. Since $\tau$ was taken
to be as small as possible, this is a contradiction unless $\tau=0$. We thus 
actually obtain such a family of disks $\Dye_t^\prime$ for $t\in [0,1]$. However, 
$\Dye_0^\prime$ is then a disk in $(\CP_2 , \RP^2)$ which meets $\mathcal Q$
transversely in the single point 
 $p$, and we have already observed that this implies that $\Dye_0^\prime$ is
one hemisphere of the projective line joining $p$ to ${\mathfrak c}(p)$. Thus, if
there were two disks $\Dye$ and $\Dye^\prime$ in $(\CP_2 , N)$,
each meeting $\mathcal Q$ transversely in the single point $p$, both could
be evolved backwards in $t$ to obtain the same disk. However, the process 
of evolving forward in $t$ starting with an initial disk at $t=0$ necessarily 
yields a unique final disk
when $t=1$, so it follows that  $\Dye^\prime=\Dye$.
We have thus succeeded in establishing uniqueness, and our proof
is therefore complete. \end{proof}

 For clarity's sake, it is worth emphasizing that the above perturbation argument is carried out
on the level of unparameterized embedded disks. For each unparameterized 
embedded disk, there is of course  an $SL(2, \RR)$'s worth of different parameterizations.
 The interested reader is invited to double-check our perturbation argument  
 using the more popular machinery of 
 parameterized disks and the  total Maslov index \cite{forst,mcsalt}.
From this perspective, the moduli space of parameterized disks in $(\CP_2 , N_t)$ 
 near a given one
 will  be $5$-dimensional, and the moduli space of parameterized
 disks passing through $p$ will be $3$-dimensional. One can furthermore 
 specialize the parameterization by requiring that $0\in D^2$ be 
 sent to $p$, and the resulting parameterized disk $f: D^2\to \CP_2$
 will then be unique modulo  rotations $f(\zeta) \rightsquigarrow f(e^{i\phi}\zeta)$.

\begin{thm} \label{crackle}
Let $N \subset \CP_2$ be any docile surface,
and let $M$ denote 
the moduli space of all holomorphic disks 
in $(\CP_2, N)$ which 
represent the generator of
$H_2 (\CP_2, N)\cong \ZZ$. Then $M$
is diffeomorphic to $S^2$. The interiors of
these disks foliate $\CP_2 - N$, and the intersection pattern 
of their boundaries defines a unique Zoll projective structure $[\nabla]$ 
on $M$. Moreover, the reference
conic $\mathcal Q$ induces a specific conformal structure $[g]$ 
on $M$, and there is a unique  $\triangledown \in [\nabla]$
which is a Weyl connection for the conformal class $[g]$. 
\end{thm} 
\begin{proof}
Theorem \ref{snap} asserts that, through each $p\in {\mathcal Q}$,
 there is a unique holomorphic disk representing the generator of 
 $H_2(\CP_2, N)$; moreover,  the proof  of this theorem shows 
 the map $\varkappa : M\to {\mathcal Q}$, obtained by  sending
 a disk to its intersection with the conic, to be a local diffeomorphism. 
Hence $\varkappa$ is actually a diffeomorphism, and 
 $M\approx {\mathcal Q}\approx S^2$. 
 
For each $x\in M$, let $\Dye_x\subset \CP_2$ be the 
embedded holomorphic disk it represents, and set
$${\mathcal F} = \{ (x,y) \in M  \times  \CP_2 ~|~ y\in \Dye_x\}.$$
Then the projection $(x,y) \mapsto x$ makes ${\mathcal F}$ into
a smooth family of disks ${\zap p} : {\mathcal F}\to M$.
Let ${\mathcal B}$ be the Melrose blow-up of $\CP_2$
along $N$, which  is the manifold-with-boundary obtained 
from $\CP_2$ by replacing each point 
of $N$ with the unit circle bundle in the normal bundle of $N$;
let ${\zap b} : {\mathcal B}\to \CP_2$ be canonical projection, 
which we shall call the {\em  Melrose blow-down}. 
Since each of the disks  $\Dye_x$ is smoothly embedded, with 
$\partial \Dye_x\subset N$, the tautological smooth projection
${\zap q}: {\mathcal F}\to \CP_2$ given by $(x,y)\mapsto y$ can 
be lifted to a smooth map $\hat{\zap q}: {\mathcal F}\to {\mathcal B}$,
with ${\zap q} = {\zap b}\circ \hat{\zap q}$,  by 
sending each boundary point to the radial derivative of ${\zap q}$ relative to 
the corresponding disk. 
However, 
since each disk $\Dye_x$ has normal Maslov index $1$, the
 normal bundle of $\Dye_x$ can be trivialized in such a manner that 
elements of the tangent space $T_xM$ are represented  
by  some holomorphic function $\varsigma : D^2\to \CC$ 
of the form 
$$
\varsigma (\zeta) = a + \bar{a} \zeta
$$
for $a\in \CC$ arbitrary. For $a\neq 0$, such a variation 
has a zero only at the boundary of the disk,
and at this zero, its radial derivative is non-zero. Thus the derivative of $\hat{\zap q}$ has maximal
rank everywhere. Since the manifolds-with-boundary in question are  compact
and connected, it follows that  $\hat{\zap q}$ is a covering map. However,
$\mathcal B$ diffeomorphic to the complement of a tubular neighborhood
of $N$,  and   ${\mathcal Q}\hookrightarrow {\mathcal B}$ is therefore 
a homotopy equivalence
by the proof of 
Lemma \ref{secondo}. Thus  ${\mathcal B}$ is simply connected, and 
 $\hat{\zap q}$ is therefore a diffeomorphism. In particular, it induces a diffeomorphism
 between the interiors of ${\mathcal F}$ and $\mathcal B$, so the interiors of 
 the disks $\Dye_x$
foliate $\CP_2 - N$, as claimed.

For
any $y\in {\mathcal B}$ and $z={\zap b}(y)$,  
consider the pull-back map ${\zap b}^*: \Lambda^{1,0}_z \to \Lambda^1_y\otimes \CC$. 
If $y$ is an interior point of $\mathcal B$, this is obviously injective, because 
${\zap b}$ is a local diffeomorphism near $y$. However, it remains injective even
if $y$ is a boundary point. To see this, notice that,  when $y\in \partial {\mathcal B}$, 
the kernel of the
 pull-back map $\Lambda^{1}_z \otimes \CC \to \Lambda^1_y\otimes \CC$
 is $1$-dimensional, and spanned by a real co-vector. Since $\Lambda^{1,0}_z\subset 
 \Lambda^{1}_z \otimes\CC$ contains no non-zero real vectors, it follows that 
 $\Lambda^{1,0}_z \to \Lambda^1_y\otimes \CC$ is injective, as claimed. 
Because ${\zap q} = {\zap b}\circ \hat{\zap q}$, the  annihilator of  ${\zap b}^*(\Lambda^{1,0})$
therefore corresponds, via the diffeomorphism $\hat{\zap q}$, to a
rank-$2$ sub-bundle $\textcyr{D} \subset T_\CC{\mathcal F}$, explicitly
given by the kernel of  ${\zap q}^{1,0}_* : T_\CC{\mathcal F} \to T^{1,0}\CP_2$.
This sub-bundle is involutive on the interior of ${\mathcal F}$, and so, by continuity,  is 
involutive even along $\partial \mathcal F$. However, since the derivative of 
${\zap b}$  has rank $3$ at every boundary point, the same is true of
${\zap q}$, and $\textcyr{D}= \ker {\zap q}^{1,0}_*$ therefore contains a
real direction at every point of $\partial {\mathcal F}$. It also contains the
$(0,1)$-tangent space of the fiber disks of ${\zap p} : {\mathcal F}\to M$, 
so we can apply 
\cite[Lemma 4.6]{lmzoll} to the double ${\mathcal F}\cup \overline{\mathcal F}$,
exactly as in the proof of \cite[Theorem 4.7]{lmzoll}. Thus, there is a unique 
projective connection $[\nabla]$ on $\mathcal M$ for which the closed curves
${\mathfrak C}_z$,
defined for  $z\in N$ by 
$${\mathfrak C}_z = {\zap p} ( {\zap q}^{-1} [\{ z\} ])
= \{ x\in \mathcal M~|~ z\in \partial \Delta_x\},$$
are the geodesics of $[\nabla]$. Moreover, since 
each disk-boundary $\partial\Delta_x$ 
is an embedded $S^1\subset N$, 
no fiber of ${\zap p}|_{\partial {\mathcal F}}$ can meet a fiber of 
${\zap q}|_{\partial {\mathcal F}}$ in more than one point. 
 Hence  the ${\mathfrak C}_z\subset {\mathfrak M}$
are all simple closed curves , and  
$[\nabla]$ is therefore a Zoll projective structure on $M\approx S^2$. 
Finally,  there is \cite[footnote 4, p. 514]{lmzoll}  a unique choice
of $\triangledown\in [\nabla]$ which is a Weyl connection for the 
conformal structure on $M$ induced by its identification with 
$\mathcal Q$. 
\end{proof}

\begin{thm} \label{pop}
Let $N \subset \CP_2$ be a docile surface,
and let $( [g], \triangledown)$ be the 
Weyl structure on 
$M\approx S^2$ whose existence is guaranteed  by Theorem \ref{crackle}. 
Then $\triangledown$ is the Levi-Civita connection 
of a Zoll metric $g\in [g]$  iff $N\subset \CP_2 - N$
is Lagrangian with respect to the sign-ambiguous
symplectic form $\Omega = \Im m \Upsilon$, where
$\Upsilon$ is defined  by equation (\ref{oops}). 
When this happens, there is a unique such $g$ whose
closed geodesics all have length $2\pi$. This normalization is moreover 
equivalent to requiring that
$(S^2, g)$ have  
total area   $4\pi$. 
\end{thm} 
\begin{proof} By \cite[Theorem 4.8]{lmzoll}, the Lagrangian condition is
equivalent to the requirement that $\triangledown$ be 
the Levi-Civita connection of a  Zoll metric $g$. 
When this happens, $g$ is then of course  determined
up to an  multiplicative constant, since 
 $\triangledown$ and $[g]$ are already known. 
We can moreover fix this scaling constant by specifying the length of some
(and hence every)  closed geodesic. 
However, it is also known \cite[Theorem 2.15]{lmzoll} that the  geodesic flow of $g$ is 
differentiably conjugate
to that of the standard metric $g_0$ on $S^2$. Hence 
 Weinstein's theorem \cite{weinstein} 
predicts that the total area of $(S^2, g)$ will coincide with that
of  $(S^2, g_0)$ iff the closed 
geodesics of $g$ and $g_0$  have identical lengths. 
\end{proof}

Theorem \ref{crackle}  considerably clarifies
our understanding of the return journey from totally real
surfaces $N\subset \CP_2$ to Zoll projective structure $[\nabla ]$
on $S^2$.  As long as 
$N$ is docile with respect to some non-singular conic $\mathcal Q$, 
then there is a unique  Zoll projective structure $[\nabla ]$ on $S^2$ 
to which $N$ corresponds via  the construction of \cite{lmzoll}. 
However,  because  docility is an open condition, 
a surface $N$ which is docile with respect to a particular conic $\mathcal Q$ will
also be docile with respect to a $5$-complex-parameter family of nearby conics 
${\mathcal Q}^\prime$. Thus, every Zoll projective structure $[\nabla]$ 
arising from Theorem \ref{crackle} can actually be represented
by a $10$-real-parameter family of 
Weyl connections $\triangledown$, each compatible with a 
different conformal structure. However, moving  ${\mathcal Q}^\prime$
far enough will always result in  conics with respect to which a 
given  $N$ will fail to be docile. For this reason, docility  
is best understood in terms of   Weyl connections rather than projective structures.

One  motivation for the present article stems from  the fact that
a Weyl structure $(S^2, [g], \triangledown)$  with  
all geodesics geometrically closed may be treated as time-symmetric 
Cauchy data for a 
Lorentzian Einstein-Weyl structure
on $S^2 \times \RR$.  From the twistor perspective, the latter arises as the moduli space of 
holomorphic disks in $\CP_1\times \CP_1$ with boundaries in 
 $\tilde{N}= \Pi^{-1}(N)$. When $N$ is docile with respect to
 the branch locus ${\mathcal Q}$ of $\Pi$, our results from  \cite{lmew}
 guarantee  long-time existence of the solution, and allow
 one to directly interpret the orientation-reversing involution 
 $\varphi : \CP_1\to \CP_1$ in terms of the scattering of light-rays
  in the resulting space-time. In this context, it would be fascinating to obtain 
a better understand the scattering map $\varphi$ directly
  in terms of the initial-value surface. 
An attack on this initial-value problem by more direct methods might therefore 
 offer new insights into the geometry of Zoll manifolds. 
 
Of course, the most classical aspects of  our present subject pertain not to  Zoll 
projective structures, but rather to  Zoll metrics. Since 
the conformal structure is  intrinsically part of the geometry in this setting,  it is completely 
natural in this context  to fix a  conic  
${\mathcal Q}\subset \CP_2$ once and for all. Our previous work in 
\cite{lmzoll} showed that every  Zoll metric on $S^2$ gives rise to 
 a Lagrangian surface $N\subset \CP_2-{\mathcal Q}$, 
one might perhaps hope that this totally real submanifold would always turn
out to be docile with respect to $\mathcal Q$. 
However, our  own calculations have revealed that this   isn't even  true
in the axisymmetric case. 
Thus, while we  hope  that  the present paper offers some
interesting  advances in the theory of Zoll metrics, further new ideas and results will 
be needed in order, for example,   to determine whether the 
 space of Zoll metrics on $S^2$ is connected. 

\bigskip

 \begin{acknow}
 The first author   warmly  thanks Chris Bishop and Dennis Sullivan 
 for  their  helpful explanations of various relevant    aspects  of the 
 theory of   quasi-conformal mappings. 
 \end{acknow}

\vfill

\noindent
{\footnotesize \sc Authors' addresses:}

 \noindent
{\footnotesize Mathematics Department, SUNY, Stony Brook, NY 11794, USA\\
Mathematical Institute, 24--29 St. Giles, Oxford OX1 3LB, UK
}


\bigskip 


\begin{thebibliography}{10}

\bibitem{alibaro}
{\sc S.~Alinhac, M.~S. Baouendi, and L.~P. Rothschild}, {\em Unique
  continuation and regularity at the boundary for holomorphic functions}, Duke
  Math. J., 61 (1990), pp.~635--653.

\bibitem{beszoll}
{\sc A.~L. Besse}, {\em Manifolds {A}ll of {W}hose {G}eodesics {A}re {C}losed},
  Springer-Verlag, Berlin, 1978.

\bibitem{chirka}
{\sc E.~M. Chirka}, {\em Regularity of the boundaries of analytic sets}, Mat.
  Sb. (N.S.), 117(159) (1982), pp.~291--336, 431.

\bibitem{forst}
{\sc F.~Forstneri{\v{c}}}, {\em Analytic disks with boundaries in a maximal
  real submanifold of {${\bf C}^2$}}, Ann. Inst. Fourier (Grenoble), 37 (1987),
  pp.~1--44.

\bibitem{grofrau}
{\sc U.~Frauenfelder}, {\em Gromov convergence of pseudoholomorphic disks}, J.
  Fixed Point Theory Appl., 3 (2008), pp.~215--271.

\bibitem{funk}
{\sc P.~Funk}, {\em {\"U}ber {F}l\"achern mit lauter geschlossenen
  geod\"atischen {L}inien}, Math. Ann., 74 (1913), pp.~278--300.

\bibitem{grezoll}
{\sc L.~W. Green}, {\em Auf {W}iedersehensfl\"achen}, Ann. of Math. (2), 78
  (1963), pp.~289--299.

\bibitem{grogro}
{\sc D.~Gromoll and K.~Grove}, {\em On metrics on {$S\sp{2}$} all of whose
  geodesics are closed}, Invent. Math., 65 (1981/82), pp.~175--177.

\bibitem{grozi}
{\sc K.~Grove and W.~Ziller}, {\em Cohomogeneity one manifolds with positive
  {R}icci curvature}, Invent. Math., 149 (2002), pp.~619--646.

\bibitem{guillzoll}
{\sc V.~Guillemin}, {\em The {R}adon transform on {Z}oll surfaces}, Advances in
  Math., 22 (1976), pp.~85--119.

\bibitem{lebrsb}
{\sc C.~LeBrun}, {\em Twistors, holomorphic disks, and {R}iemann surfaces with
  boundary}, in Perspectives in Riemannian geometry, vol.~40 of CRM Proc.
  Lecture Notes, Amer. Math. Soc., Providence, RI, 2006, pp.~209--221.

\bibitem{lmzoll}
{\sc C.~LeBrun and L.~J. Mason}, {\em Zoll manifolds and complex surfaces}, J.
  Differential Geom., 61 (2002), pp.~453--535.

\bibitem{lmew}
\leavevmode\vrule height 2pt depth -1.6pt width 23pt, {\em The
  {E}instein-{W}eyl equations, scattering maps, and holomorphic disks}, Math.
  Res. Lett., 16 (2009), pp.~291--301.

\bibitem{lehvir}
{\sc O.~Lehto and K.~I. Virtanen}, {\em Quasiconformal mappings in the plane},
  Springer-Verlag, New York, second~ed., 1973.

\bibitem{mcsalt}
{\sc D.~McDuff and D.~Salamon}, {\em {$J$}-holomorphic curves and symplectic
  topology}, American Mathematical Society, Providence, RI, 2004.

\bibitem{rickman}
{\sc S.~Rickman}, {\em Quasiregular mappings}, Springer-Verlag, Berlin, 1993.

\bibitem{rochon}
{\sc F.~Rochon}, {\em On the uniqueness of certain families of holomorphic
  disks}.
\newblock e-print arXiv:0709.1118 [math.DG], 2007.

\bibitem{weinstein}
{\sc A.~Weinstein}, {\em On the volume of manifolds all of whose geodesics are
  closed}, J. Differential Geometry, 9 (1974), pp.~513--517.

\bibitem{yau}
{\sc S.~T. Yau}, {\em {C}alabi's conjecture and some new results in algebraic
  geometry}, Proc. Nat. Acad. USA, 74 (1977), pp.~1789--1799.

\bibitem{groye}
{\sc R.~Ye}, {\em Gromov's compactness theorem for pseudo holomorphic curves},
  Trans. Amer. Math. Soc., 342 (1994), pp.~671--694.

\bibitem{zoll}
{\sc O.~Zoll}, {\em {\"U}ber {F}l\"achen mit {S}charen geschlossener
  geod\"atischer {L}inien}, Math. Ann., 57 (1903), pp.~108--133.

\end{thebibliography}
  \end{document}